\documentclass[12pt]{amsart}
\usepackage{amssymb}
\usepackage{amsmath, amscd}
\usepackage{amsthm}
\usepackage{xcolor}

\newtheorem{theorem}{Theorem}[section]

\newtheorem{lemma}[theorem]{Lemma}

\theoremstyle{definition}

\newcommand{\Z}{\mathbb{Z}}

\begin{document}

\author[P.V. Danchev]{Peter V. Danchev}
\address{Institute of Mathematics and Informatics, Bulgarian Academy of Sciences, 1113 Sofia, Bulgaria}
\email{danchev@math.bas.bg; pvdanchev@yahoo.com}
\author[B. Goldsmith]{Brendan Goldsmith}
\address{Technological University, Dublin, Dublin 7, Ireland}
\email{brendan.goldsmith@TUDublin.ie; brendangoldsmith49@gmail.com}
\author[P.W. Keef]{Patrick W. Keef}
\address{Department of Mathematics, Whitman College, Walla Walla, WA 99362, USA}
\email{keef@whitman.edu}

\title[Hereditary and Super Properties in co-Bassian Groups] {Hereditarily and Super Generalized co-Bassian Abelian Groups}

\maketitle

\begin{abstract} We completely characterize by finding necessary and sufficient conditions those co-Bassian and generalized co-Bassian Abelian groups having, respectively, the hereditary or the super property, thus giving a new insight in the full discovery of the structure of these two classes of groups as recently defined in Arch. Math. Basel (2024) by the third author.
\end{abstract}

\vskip2.0pc

\section{Introduction and Background}
	
All groups considered in this short article are additively written and Abelian, and the standard notations are mainly in agreement with those from \cite{F1} and \cite{F}. The more specific terminology is recollected as follows: mimicking \cite{K}, a group $G$ is called {\it co-Bassian} if, for all $N\leq G$, if $\phi:G\to G/N$ is an injection, then $\phi(G)=G/N$. And $G$ is called {\it generalized co-Bassian} if, for all $N\leq G$, if $\phi:G\to G/N$ is an injection, then $\phi(G)$ is a direct summand of $G/N$. Thereby, all co-Bassian groups are generalized co-Bassian. However, the converse implication is {\it not} true; even more, any co-Bassian group is co-Hopfian, but there is a generalized co-Bassian group that is {\it not} co-Hopfian (for a more detailed information, we refer to the related sources \cite{CDK} and \cite{GG1}, \cite{GG2}).

\medskip

Our objective here is to give a more deep penetrating in the complete description of the defined above co-Bassian and generalized co-Bassian groups by exploring their behavior with respect subgroups and factor-groups. So, a group is said to be {\it hereditarily co-Bassian} (resp., {\it hereditarily generalized co-Bassian}) if all its subgroups are co-Bassian (resp., generalized co-Bassian). Likewise, a group is said to be {\it super co-Bassian} (resp., {\it super generalized co-Bassian}) if all its factor-groups or, equivalently, all its epimorphic images, are co-Bassian (resp., generalized co-Bassian).

\medskip

Our work is based on the following two characterizations from \cite{K}:

\begin{theorem}\label{Keef1} (\cite[Theorem 2.6]{K})
The group $G$ is co-Bassian if, and only if, for each prime $p$, $T_p$ has finite $p$-rank and $G/T$ is a divisible group of finite torsion-free rank.
\end{theorem}

\begin{theorem}\label{Keef2} (\cite[Theorem 2.5]{K})
The group $G$ is generalized co-Bassian if, and only if, one of the following two conditions holds:

(a) $G$ is divisible;

(b) $G/T$ is torsion-free divisible of finite rank and, for each prime $p$, $T_p$ has generalized finite $p$-rank.
\end{theorem}

In this aspect, the organization of our further work is planned thus: In the next section, we state and prove our main results distributed into two subsections, which treated the hereditary case and the super case, respectively. We establish four theorems concerning these two properties for the co-Bassian and generalized co-Bassian groups (see, respectively, Theorems~\ref{first} and \ref{second} as well as Theorems~\ref{third} and \ref{fourth}).

\section{Principal Results}

Our basic results are stated and proved in the next two subsequent subsections.

\subsection{The Hereditary Property in Generalized co-Bassian Groups}

The first of our two main establishments in this section is the following one.

\begin{theorem}\label{first} For a group $G$, the following three conditions are equivalent:

(a) $G$ is hereditarily co-Bassian;

(b) $G$ does not have a subgroup $A\leq G$ such that $A\cong \Z$ or $A\cong \Z(p)^{(\omega)}$ for any prime $p$;

(c) $G$ is a torsion group such that $T_p$ is finitely co-generated for all primes $p$.
\end{theorem}

\begin{proof} Unsurprisingly, the following argument is based up Theorem~\ref{Keef1}.

Suppose (b) fails; we verify that (a) must fail, as well. If $A\leq G$ with $A\cong \Z$, then since $\Z$ is not co-Bassian, $G$ must fail to be hereditarily co-Bassian. Similarly, since $A\cong \Z(p)^{(\omega)}$ also fails to be co-Bassian, $G$ cannot have a subgroup of that form either, as expected.

Suppose now (b) holds. Note that $G$ is a torsion group if, and only if, it does not have a subgroup $A\leq G$ such that $A\cong \Z$. And the $p$-subgroup $T_p$ is finitely co-generated if, and only if, $G$ does not have a subgroup  $A\leq G$ such that $A\cong \Z(p)^{(\omega)}$, which shows that (b) is equivalent to (c), as promised.

Finally, if (c) holds, then any subgroup $H\leq G$ will also be torsion and, if $S_p$ is the $p$-torsion subgroup of $H$, then $S_p\leq T_p$ shows that $S_p$ is also finitely co-generated. Therefore, all such $H$ will also be co-Bassian, so that $G$ is hereditarily co-Bassian, as stated.
\end{proof}

Our next major achievement is the following criterion.

\begin{theorem}\label{second} For a group $G$, the following three conditions are equivalent:

(a) $G$ is hereditarily generalized co-Bassian;

(b) $G$ does not have a subgroup $A\leq G$ such that $A\cong \Z$ or $A\cong \Z(p^2)^{(\omega)}$ for any prime $p$;

(c) $G$ is a torsion group such that $pT_p$ is finitely co-generated for all primes $p$.
\end{theorem}

\begin{proof} We utilize in our argumentation the characterization in Theorem~\ref{Keef2}.

Suppose (b) fails. If $\Z\cong A\leq G$, then $A$ is not generalized co-Bassian, so that $G$ is not hereditarily generalized co-Bassian. Next, if $$\Z(p^2)^{(\omega)}\cong A\leq G,$$ then it is clear that there is a subgroup $A'\leq A$ such that $$A'\cong \Z(p)\oplus \Z(p^2)^{(\omega)}.$$ However, this $A'$ will not have generalized finite $p$-rank, so that it is not generalized co-Bassian, so that $G$ again fails to be hereditarily generalized co-Bassian, as asked for.

Suppose now (b) holds. Note that $G$ is a torsion group if, and only if, it does not have a subgroup $A\leq G$ such that $A\cong \Z$. And $pT_p$ is finitely co-generated if, and only if, $(pT_p)[p]$ is finite if, and only if, $G$ does not have a subgroup $A\leq G$ such that $A\cong \Z(p^2)^{(\omega)}$, which shows that (b) is equivalent to (c), as formulated.

Finally, if (c) holds, then any subgroup $H\leq G$ will also be torsion and, if $S_p$ is the $p$-torsion subgroup of $H$, then $pS_p\leq pT_p$ shows that $pS_p$ is also finitely co-generated. Therefore, all such $H$ will also be generalized co-Bassian, so that $G$ is hereditarily generalized co-Bassian, as wanted.
\end{proof}

\subsection{The Super Property in Generalized co-Bassian Groups}

We start with some elementary observations about $p$-ranks. If $G$ is group with torsion part $T$, then, for all primes $p$, we know that $G$, $T$ and $T_p$ all have the same $p$-rank. 

\medskip

We also prove the following technicality.

\begin{lemma}\label{ranks}
Suppose $G$ is a group of finite $p$-rank for some prime $p$ and finite torsion-free rank. If $K\leq G$ and $H=G/K$, then $H$ also has finite torsion-free rank and finite $p$-rank.
\end{lemma}

\begin{proof}
It is obvious that $H$ has finite torsion-free rank. Let $S$ be the torsion subgroup of $H$; we need to show $S_p[p]$ is finite. There is a sequence
$$
0 \to K[p]\to T_p[p]\to S_p[p]\to K/pK .
$$
We claim that $K/pK$ is finite: to that end, let $K_{(p)}$ be the localization at $p$, so that $$K/pK\cong K_{(p)}/pK_{(p)}.$$ If $R$ is the torsion subgroup of $K_{(p)}$, it follows that $R$ has finite $p$-rank whence $K_{(p)}=R\oplus X$ for some torsion-free $X$ of finite torsion-free rank. Therefore, it must be that
$$
          K_{(p)}/pK_{(p)}\cong (R/pR)\oplus (X/pX),
$$
and since both of these latter terms are finite, so is $K/pK\cong K_{(p)}/pK_{(p)}$.

Thus, since $K[p]$ and $T_p[p]$ are finite, it follows that $S_p[p]$ must also be finite, as desired.
\end{proof}

The first of our two main establishments in this section is the following one. It is again based upon Theorem~\ref{Keef1}.

\begin{theorem}\label{third} A group is super co-Bassian if, and only if, it is co-Bassian.
\end{theorem}

\begin{proof} Assume $G$ is co-Bassian, and suppose $\pi:G\to H$ is an epimorphism with kernel $K$ and $S$ is the torsion subgroup of $H$. Clearly, $H/S$ will be divisible of finite rank, and Lemma~\ref{ranks} applies to get that, for all prime $p$, $H$ has finite $p$-rank. Consequently, $H$ is generalized co-Bassian and so in view of Theorem~\ref{Keef1} the group $G$ is super generalized co-Bassian, as claimed.

The converse claim is obvious, so we are done.
\end{proof}

Our next major achievement is the following criterion. Of course, it is based upon Theorem~\ref{Keef2}.

\begin{theorem}\label{fourth} The group $G$ is super generalized co-Bassian if, and only if, one of the following two conditions holds:

(a) $G$ is divisible;

(b) $G/T$ is torsion-free divisible of finite rank and, for each prime $p$, $T_p$ is either divisible, or $pT_p$ has finite $p$-rank.
\end{theorem}

\begin{proof} Suppose first that $G$ is super generalized co-Bassian. In particular, it is generalized co-Bassian, so in virtue of Theorem~\ref{Keef2} it is either divisible (and thus (a) holds), or else $G/T$ is divisible of finite rank and, for each prime $p$, $T_p$ has generalized finite $p$-rank.  If, however, we assume (b) failed, then one knows that, for some prime $p$, the subgroup $T_p$ (and hence the whole group  $G$) would have a direct summand of the form $X=\Z(p^n)^{(\omega)}$, where $1<n<\infty$. So, if we can show $X$ is not super generalized co-Bassian, the same will hold for $G$, as required.

Note that, if $Y=\Z(p)\oplus X$, then it is easy to see that there is an epimorphism $X\to Y$, but since $Y$ does not have generalized finite $p$-rank, it cannot be generalized co-Bassian referring to Theorem~\ref{Keef2}. Therefore, $X$, and hence $G$, cannot be super generalized co-Bassian, contrary to our assumption.

Conversely, suppose $G$ satisfies (a) or (b). If it is divisible, then any epimorphic image of $G$ is also divisible, and hence generalized co-Bassian, i.e., a divisible group is evidently super generalized co-Bassian.

So, assume $G$ satisfies (b), $K\leq G$ and $H=G/K$ with torsion $S$. If $T_p$ is divisible, it follows that $T$ is $p$-divisible. And since $G/T$ is divisible, it is also $p$-divisible. Consequently, $G$ must be $p$-divisible, so that the same holds for $H$. Therefore, $S_p$ must be divisible, an hence of generalized finite $p$-rank.

On the other hand, if $pT_p$ has finite $p$-rank, then it is easily seen that there is an epimorphism $pG\to pH$. Invoking Lemma~\ref{ranks}, $pH$ will have finite $p$-rank, so that $pS$ does, as well.

It, thus, follows that each $S_p$ will have generalized finite $p$-ranks for all primes $p$. That is why, all such epimorphic images $H$ of $G$ are themselves generalized co-Bassian, so that $G$ is super generalized co-Bassian, as asserted.
\end{proof}





	

\end{document}